\documentclass[11pt]{article}

\usepackage[a4paper,margin=1in]{geometry}
\usepackage{amsmath,amssymb,amsthm,mathtools}
\usepackage{enumitem}
\usepackage{hyperref}

\hypersetup{
  colorlinks=true,
  linkcolor=blue,
  citecolor=blue,
  urlcolor=blue
}

\numberwithin{equation}{section}

\newtheorem{theorem}{Theorem}[section]
\newtheorem{proposition}[theorem]{Proposition}
\newtheorem{lemma}[theorem]{Lemma}
\newtheorem{corollary}[theorem]{Corollary}
\newtheorem{definition}[theorem]{Definition}
\newtheorem{remark}[theorem]{Remark}

\newcommand{\T}{\mathbb T}
\newcommand{\R}{\mathbb R}
\newcommand{\Id}{I}
\newcommand{\Spp}{\mathbb S_{++}}
\newcommand{\Sym}{\mathbb S}
\newcommand{\tr}{\operatorname{tr}}
\newcommand{\divv}{\operatorname{div}}
\newcommand{\curl}{\operatorname{curl}}
\newcommand{\Log}{\operatorname{Log}}

\newcommand{\Besov}{B^0_{\infty,1}}
\newcommand{\calE}{\mathcal E}
\newcommand{\calM}{\mathcal M}
\newcommand{\calN}{\mathcal N}
\newcommand{\calS}{\mathcal S}
\newcommand{\calR}{\mathcal R}
\newcommand{\calA}{\mathcal A}
\newcommand{\norm}[1]{\left\lVert #1\right\rVert}

\title{Three-Dimensional Positive-Cone Oldroyd--B Flows:\\
Geometric Continuation and Residual-Work Criteria}

\author{Sai Peng\\
School of Mathematics and Computational Science, Xiangtan University\\
\texttt{pscfd@xtu.edu.cn}}
\date{}

\begin{document}
\maketitle

\begin{abstract}
We prove a three-dimensional positive-cone continuation criterion for the
stress-diffusion-free Oldroyd--B system on the periodic torus.  Writing the
positive conformation tensor as \(A=e^B\), finite-time breakdown of a strong
\(H^s\), \(s>5/2\), solution can occur only through loss of the logarithmic
spectral envelope \(\|\Log A\|_{L^\infty}\) or divergence of the endpoint
vorticity clock
\[
  \int_0^{T_*}\|\operatorname{curl}u(t)\|_{B^0_{\infty,1}}\,dt .
\]
The proof uses compact positive-cone envelopes, endpoint Biot--Savart
estimates, and high-order logarithmic conformation estimates, without stress
diffusion.  We also derive a positive-cone Reynolds admissibility criterion
with an exact cost: the least \(L^2\) conformation residual needed to pay positive
pressure-free residual work is determined by the entropy-dual lever
\(G=I-A^{-1}\), and the cost degenerates quantitatively at \(A=I\).
\end{abstract}



\noindent\textbf{Short title.}
3D Positive-Cone Oldroyd--B Flows.

\noindent\textbf{Correspondence.}
Sai Peng, School of Mathematics and Computational Science, Xiangtan
University, Xiangtan, China.
Email: \texttt{pscfd@xtu.edu.cn}.

\section*{Notation}

We write \(\Sym^3\) for real symmetric \(3\times3\) matrices and
\(\Spp^3\subset\Sym^3\) for the positive definite cone.  The Frobenius product
is \(M:N=\tr(MN^T)\).  The strain tensor is
\[
  D(u)=\frac12(\nabla u+(\nabla u)^T).
\]
All function spaces are periodic on \(\T^3\).  The inhomogeneous endpoint
Besov norm is
\[
  \norm{f}_{B^0_{\infty,1}}
  =
  \sum_{j\ge -1}\norm{\Delta_jf}_{L^\infty},
\]
with respect to a fixed smooth Littlewood--Paley partition.  Different
partitions give equivalent norms.  We write \(C\) for harmless constants and
\(C_{K,s}\) for constants depending on the spectral envelope
\(\norm{\Log A}_{L^\infty}\le K\) and on \(s\).  Constants may also depend on
the fixed parameters \(\nu,\alpha,\lambda\) and on the chosen flat torus, but
not on the time of continuation.

The pressure is always understood modulo functions of time and may be
normalized to have zero spatial mean.  The spatial mean of the velocity is
conserved by the periodic equations and is kept fixed by the initial data.
The Biot--Savart operator is applied after subtracting the conserved mean of
\(u\), which does not change \(\nabla u\).  The anti-divergence \(\calA\) is
used only for zero-mean vector fields; in the residual-work construction this
zero-mean condition follows from the fixed spatial mean of \(u\).

\section{Introduction}

The incompressible Oldroyd--B system without stress diffusion couples a viscous
fluid velocity \(u\) to a transported, stretched, and relaxed positive
conformation tensor \(A\).  On the three-dimensional periodic torus it reads
\begin{align}
  \partial_t u+u\cdot\nabla u-\nu\Delta u+\nabla p
  &=\alpha\divv(A-\Id),\qquad \divv u=0,
  \label{eq:momentum}\\
  \partial_t A+u\cdot\nabla A
  &=\nabla u\,A+A(\nabla u)^T-\lambda^{-1}(A-\Id),
  \qquad A(t,x)\in\Spp^3 .
  \label{eq:conf}
\end{align}
Here \(\nu,\alpha,\lambda>0\).  The absence of stress diffusion leaves the
conformation equation hyperbolic in space.  Regularity is therefore not a
question of free energy alone: one must identify whether transport and
stretching can create high-frequency concentration while the tensor remains in
the positive cone.

The natural variable for the cone is the logarithmic conformation
\[
  B=\Log A .
\]
It records both upper and lower spectral control:
\[
  \norm{B}_{L^\infty}\le K
  \quad\Longleftrightarrow\quad
  e^{-K}\Id\le A\le e^K\Id .
\]
This viewpoint follows the two-dimensional positive-cone criterion for
Oldroyd--B and FENE-P flows in \cite{PengUnified}, where continuation is
separated into flow-map control and logarithmic conformation concentration,
with an additional trace-gap barrier in the FENE-P case.  In three dimensions
the cone geometry is unchanged, but the velocity side is no longer a scalar
vorticity problem.  The first result of the paper is a geometric
Beale--Kato--Majda type criterion: on a compact logarithmic spectral envelope,
the endpoint vorticity clock controls the strong norm.  The key
three-dimensional input is that the Biot--Savart law still gives the endpoint
estimate
\[
  \norm{\nabla u}_{L^\infty}\lesssim
  \norm{\curl u}_{B^0_{\infty,1}} .
\]
The result is not a global regularity theorem: the three-dimensional vortex
stretching term remains present.  It identifies the endpoint quantity that any
large-data regularity argument for the stress-diffusion-free system would still
have to control.  The use of \(B^0_{\infty,1}\) is not only a technical choice.
The continuation argument needs an integrable Lipschitz modulus for
the Lagrangian flow and an endpoint commutator bound stable under order-zero
multipliers.  Vorticity in \(L^\infty\) does not control \(\nabla u\) after
Riesz transforms, while a \(BMO\)-type bound lacks the dyadic summability needed
in the transport estimates.  The Besov clock is therefore the natural endpoint
scale at which Biot--Savart, transport, and logarithmic conformation estimates
close simultaneously.

The second result concerns residual work for positive-cone Oldroyd--B Reynolds
states.  Once positivity is built into the variable \(A=e^B\), momentum
residuals and conformation residuals play different roles.  The momentum
residual is flexible, up to the mean mode, through a symmetric trace-free
anti-divergence.  The conformation residual is constrained by the
energy--entropy inequality and must perform signed work against
\[
  G(A)=\Id-A^{-1}.
\]
This produces an exact Hilbert-space minimization formula for the least
\(L^2\) conformation residual required to pay positive pressure-free residual
work.  In particular the cost degenerates at the cone tip \(A=\Id\), where the
entropy-dual lever \(G\) vanishes.

\subsection*{Contribution and scope}

The paper has two aims.  It proves a three-dimensional continuation criterion
and a residual-work admissibility criterion in the stress-diffusion-free
positive cone.  These results give two necessary alternatives: a strong solution
can fail only through the endpoint flow clock or through loss of the
logarithmic cone envelope, while an energy--entropy-admissible relaxed state can
produce positive pressure-free residual work only by paying for it in the
conformation equation.  The two alternatives are different expressions of the
same obstruction.  In the strong regime, the obstruction appears as failure to
keep a Lipschitz flow on a compact logarithmic cone.  In the relaxed regime, it
appears as the impossibility of hiding positive mechanical residual work unless
the conformation equation supplies a quantitatively sufficient entropy-dual
defect.  This is the sense in which the continuation and residual criteria are
used together in this paper.

\subsection*{Related literature}

Oldroyd's constitutive framework originates in \cite{Oldroyd}; standard
background on polymeric and viscoelastic fluids may be found in
\cite{Bird,Renardy,RenardyThomases}.  Local strong solutions and lifespan
questions for Oldroyd-type systems were developed in works such as
\cite{GuillopeSaut,CheminMasmoudi}, and blow-up criteria were refined in
\cite{LeiMasmoudiZhou}.  Weak, regularized, and numerical
positive-conformation theories exploit the free-energy structure and the
logarithmic entropy; see, for example, \cite{LionsMasmoudi,BarrettBoyaval}.
In contrast to stress-diffusive models, where additional parabolic smoothing is
available \cite{ConstantinKliegl,ElgindiRousset,HuangLiuZi}, the present paper
treats the stress-diffusion-free system and therefore relies on velocity
dissipation, relaxation, and cone geometry.

Recent three-dimensional results further clarify the scope of the problem.
Global theories are known in perturbative or structurally enhanced regimes,
including small solutions without stress damping \cite{Zhu2018}, analytic or
vanishing-viscosity regimes \cite{Zi2021}, mixed partial dissipation
\cite{LinWuBoardman}, inviscid damping limits \cite{ChengLuoYangYuan}, and
stress-dissipative large-data classes \cite{LiangLiZhai}.  These advances do
not close the large-data, three-dimensional, stress-diffusion-free positive-cone
problem treated here.  The present criterion is complementary: it isolates the
geometric endpoint quantities that any such closure must control.

Log-conformation variables are classical in high-Weissenberg computations
\cite{FattalKupferman}; here they are used analytically to separate spectral
degeneration from endpoint oscillation.  The continuation criterion follows
the Beale--Kato--Majda philosophy \cite{BKM}, with endpoint commutator and
Besov estimates in the spirit of \cite{KatoPonce,BCD}.  The residual
formulation is also informed by the principle that momentum Reynolds defects
are flexible whereas admissibility imposes signed work constraints, a theme
familiar from weak-solution theories for fluid equations
\cite{DeLellisSzekelyhidi}.

\subsection*{Organization}

The remainder of the introduction states the two main results.  Section~2
records the energy identity, logarithmic relaxation control, and propagation of
the positive cone.  Section~3 proves the endpoint vorticity clock estimate.
Sections~4 and~5 prove the logarithmic high-order estimates and the geometric
continuation criterion.  Section~6 records the good unknown associated with
the vorticity--stress coupling.  Sections~7--9 develop the positive-cone
Reynolds system, the sharp residual-work cost, and the cone-tip obstruction.
Section~10 states maximal-time alternatives.  Sections~11 and~12 discuss the
consequences and conclude.  The appendix collects the endpoint analytic
estimates used in the proof.
\subsection*{Main continuation result}

Throughout the continuation statements, a strong \(H^s\) solution on
\([0,T)\), \(s>5/2\), means
\[
  u\in C([0,T);H^s_\sigma(\T^3))\cap L^2_{\mathrm{loc}}([0,T);H^{s+1}),
  \qquad
  A-\Id\in C([0,T);H^s(\T^3)),
\]
with \(A(t,x)\in\Spp^3\), with the equations holding in the usual Sobolev
sense, and with the spatial mean of \(u\) fixed by the initial data.  Maximality
is always understood in this class.

\begin{theorem}[Geometric continuation criterion]
\label{thm:main}
Let \(s>5/2\), let \(u_0\in H^s(\T^3)\) be divergence free, and let
\(A_0-\Id\in H^s(\T^3)\) with \(A_0(x)\in\Spp^3\).  Let \((u,A)\) be the
maximal strong solution of \eqref{eq:momentum}--\eqref{eq:conf} on
\([0,T_*)\).  If \(T_*<\infty\), then
\begin{equation}
  \limsup_{t\uparrow T_*}\norm{\Log A(t)}_{L^\infty}
  +
  \int_0^{T_*}\norm{\omega(t)}_{\Besov}\,dt
  =\infty ,
  \label{eq:main-criterion}
\end{equation}
where \(\omega=\curl u\).  Equivalently, if a strong solution on a finite
interval \([0,T)\) satisfies
\[
  \sup_{0\le t<T}\norm{\Log A(t)}_{L^\infty}<\infty,
  \qquad
  \int_0^T\norm{\omega(t)}_{\Besov}\,dt<\infty,
\]
then the strong solution continues beyond \(T\).
\end{theorem}

\begin{corollary}[Vorticity-only endpoint obstruction]
\label{cor:vorticity-only}
Under the hypotheses of Theorem~\ref{thm:main}, finite-time breakdown implies
\[
  \int_0^{T_*}\norm{\omega(t)}_{\Besov}\,dt=\infty .
\]
\end{corollary}

The corollary follows because the endpoint vorticity clock controls
\(\int\norm{\nabla u}_{L^\infty}\,dt\), and the latter propagates the
logarithmic spectral envelope of \(A\).  Thus the spectral term in
\eqref{eq:main-criterion} is not an independent obstruction once the endpoint
vorticity clock is finite.

\begin{remark}[Regularity threshold and endpoint scale]
The assumption \(s>5/2\) is the natural strong-solution threshold for the
present proof: it gives \(H^s(\T^3)\hookrightarrow C^1\), a classical
Lagrangian flow, and a locally Lipschitz conformation equation on compact cone
envelopes.  The theorem is not intended as a critical well-posedness result at
\(s=5/2\).  Lowering the strong topology to a critical scale would require a
separate local theory compatible with the positive-cone constraint.  By
contrast, the blow-up clock itself is endpoint in the velocity variable: the
\(B^0_{\infty,1}\) norm is used only to control the Lipschitz modulus and the
order-zero commutator structure, not to impose a supercritical Sobolev bound.
\end{remark}

\subsection*{Main residual-work result}

For smooth positive-cone Reynolds states, the momentum residual is flexible:
up to the mean mode it can be represented by a symmetric trace-free
anti-divergence.  The conformation residual is not flexible in the same way,
because it must pay the energy--entropy work with the correct sign against
\[
  G(A)=\Id-A^{-1}.
\]
In the statement below \(R_{\calN}\) denotes the canonical representative
selected by the anti-divergence operator \(\calA\).  Changing the pressure
representative adds only \(p\Id\), which performs no work against \(D(u)\) for
incompressible \(u\).

\begin{theorem}[Sharp residual-work criterion]
\label{thm:residual-main}
Let \(u,p\) be smooth on a time interval, with \(\divv u=0\), and let
\(B(t,x)\in\Sym^3\) be smooth.  Set \(A=e^B\).  Define the conformation
residual and the pressure-free momentum residual by
\[
  \calS[u,B]
  =
  \partial_t A+u\cdot\nabla A-\nabla u\,A-A(\nabla u)^T
  +\lambda^{-1}(A-\Id),
\]
and
\[
  \calN[u,B]
  =
  \partial_t u+\divv(u\otimes u)-\nu\Delta u-\alpha\divv(A-\Id).
\]
Since \(B\in\Sym^3\), \(A=e^B\in\Sym^3\).  Hence
\(\partial_tA+u\cdot\nabla A\), \(\nabla u\,A+A(\nabla u)^T\), and
\(A-\Id\) are symmetric, so \(\calS[u,B](t,x)\in\Sym^3\).
Assume the spatial mean of \(u\) is fixed in time, so that \(\calN[u,B]\) has
zero spatial mean.  Let
\(R_{\calN}=\calA\calN[u,B]\) be the canonical symmetric trace-free
anti-divergence of the pressure-free momentum residual constructed in
Proposition~\ref{prop:antidiv}, and define
\[
  P(t)=\int_{\T^3}-R_{\calN}(t):D(u(t))\,dx,\qquad
  G(t)=\Id-A(t)^{-1}.
\]
Here the residual-matching Reynolds state associated with \((u,p,B)\) is the
state with conformation defect \(S=\calS[u,B]\) and momentum defect represented
by \(R_{\calN}\) modulo the harmless pressure tensor.  If this state is
energy--entropy admissible, then for a.e. \(t\),
\begin{equation}
  [P(t)]_+
  \le
  \frac{\alpha}{2}
  \norm{G(t)}_{L^2}\norm{\calS[u,B](t)}_{L^2}.
  \label{eq:main-residual-cost}
\end{equation}
Moreover, when \(G(t)\not\equiv0\), the smallest \(L^2\)-size among all
candidate conformation defects \(S(t,\cdot)\in L^2(\T^3;\Sym^3)\) satisfying
the same signed work constraint
\[
  P(t)+\frac{\alpha}{2}\int_{\T^3}G(t):S(t)\,dx\le0
\]
is
\begin{equation}
  \frac{2[P(t)]_+}{\alpha\norm{G(t)}_{L^2}},
  \label{eq:main-min-cost}
\end{equation}
and, if \(P(t)>0\), the unique minimizer in this Hilbert-space problem is
\[
  S_{\min}(t,x)
  =
  -\frac{2P(t)}{\alpha\norm{G(t)}_{L^2}^2}G(t,x).
\]
This minimizer is symmetric because \(G(t,x)\in\Sym^3\).
Thus the actual residual \(\calS[u,B](t)\), whenever admissible, must have
\(L^2\)-size at least \eqref{eq:main-min-cost}.
If \(G(t)\equiv0\), equivalently \(A(t)=\Id\) a.e., admissibility forces
\(P(t)\le0\).
\end{theorem}

The estimate is dimension-independent at the algebraic level.  The only
three-dimensional modification is the construction of the symmetric
trace-free anti-divergence.

\begin{remark}[Meaning of the residuals]
The residuals in Theorem~\ref{thm:residual-main} may be read in three related
ways.  Analytically, they describe a Reynolds-type relaxation of the exact
Oldroyd--B equations.  Numerically, they measure the momentum and conformation
defects left by an approximate scheme after projecting the momentum error away
from the pressure gauge.  Modelling-wise, they describe an unresolved closure
added to the polymer equation.  The theorem says that these interpretations
are not interchangeable: a momentum defect can be represented flexibly, whereas
an admissible conformation defect must have the correct size and sign against
the entropy-dual direction \(G=\Id-A^{-1}\).
\end{remark}

\section{Entropy and cone geometry}

Define
\[
  \Phi_3(A)=\tr A-\log\det A-3,\qquad A\in\Spp^3,
\]
and
\[
  \calE(t)=
  \frac12\int_{\T^3}|u|^2\,dx
  +\frac{\alpha}{2}\int_{\T^3}\Phi_3(A)\,dx .
\]

\begin{lemma}[Symmetry preservation]
\label{lem:symmetry}
Let \(A\) solve \eqref{eq:conf} smoothly on \([0,T]\).  If
\(A(0,x)=A(0,x)^T\), then \(A(t,x)=A(t,x)^T\) for all \(0\le t\le T\).
\end{lemma}

\begin{proof}
Let \(C=A-A^T\) and \(M=\nabla u\).  Taking the antisymmetric part of
\eqref{eq:conf} gives
\[
  \partial_tC+u\cdot\nabla C
  =
  MC+CM^T-\lambda^{-1}C .
\]
Along the flow of \(u\),
\[
  \frac{d}{dt}|C|
  \le
  (2\norm{\nabla u}_{L^\infty}+\lambda^{-1})|C|.
\]
Since \(C(0)=0\), Gronwall's lemma gives \(C\equiv0\).
\end{proof}

\begin{lemma}[Stretching cancellation]
\label{lem:cancellation}
Let \(A\in\Spp^3\) and \(M\in\R^{3\times3}\) with \(\tr M=0\).  Then
\[
  (\Id-A^{-1}):(MA+AM^T)=2(A-\Id):M .
\]
\end{lemma}

\begin{proof}
By cyclicity of trace,
\[
  (\Id-A^{-1}):MA=\tr(MA)-\tr(A^{-1}MA)=\tr(MA)-\tr M=A:M .
\]
The second term gives the same contribution.  Since \(\tr M=0\),
\(A:M=(A-\Id):M\).
\end{proof}

\begin{proposition}[Free-energy identity]
\label{prop:energy}
Every smooth solution of \eqref{eq:momentum}--\eqref{eq:conf} satisfies
\begin{equation}
  \frac{d}{dt}\calE(t)
  +\nu\int_{\T^3}|\nabla u|^2\,dx
  +\frac{\alpha}{2\lambda}
  \int_{\T^3}\tr(A+A^{-1}-2\Id)\,dx
  =0 .
  \label{eq:energy}
\end{equation}
\end{proposition}

\begin{proof}
Testing the momentum equation by \(u\) yields the kinetic balance with stress
work
\[
  -\alpha\int_{\T^3}(A-\Id):\nabla u\,dx .
\]
For the conformation part, \(D\Phi_3(A)=\Id-A^{-1}\).  The transport
contribution integrates to zero because \(\divv u=0\).  Testing the remaining
terms in \eqref{eq:conf} by \(\Id-A^{-1}\) yields
\[
  \frac{d}{dt}\int_{\T^3}\Phi_3(A)\,dx
  =
  2\int_{\T^3}(A-\Id):\nabla u\,dx
  -\lambda^{-1}
  \int_{\T^3}\tr(A+A^{-1}-2\Id)\,dx ,
\]
where Lemma~\ref{lem:cancellation} is used with \(M=\nabla u\), and the
relaxation identity follows from
\[
  (\Id-A^{-1}):(A-\Id)=\tr(A+A^{-1}-2\Id).
\]
The last integrand is nonnegative, since after diagonalizing \(A\) it becomes
\(\sum_i(a_i+a_i^{-1}-2)\).  Multiplying by \(\alpha/2\) and adding gives
\eqref{eq:energy}.
\end{proof}

\begin{proposition}[Logarithmic relaxation control]
\label{prop:log-relax}
For \(B=\Log A\),
\begin{equation}
  |B|^2\le\tr(A+A^{-1}-2\Id).
  \label{eq:pointwise-log-control}
\end{equation}
Consequently, on any interval \([0,T]\) of smooth existence,
\begin{equation}
  \int_0^T\norm{\Log A(t)}_{L^2(\T^3)}^2\,dt
  \le
  \frac{2\lambda}{\alpha}\calE(0).
  \label{eq:log-L2-time}
\end{equation}
\end{proposition}

\begin{proof}
If the eigenvalues of \(B\) are \(z_1,z_2,z_3\), then the eigenvalues of \(A\)
are \(e^{z_i}\), and
\[
  \tr(A+A^{-1}-2\Id)
  =
  2\sum_{i=1}^3(\cosh z_i-1)
  \ge \sum_{i=1}^3z_i^2=|B|^2 .
\]
The time-integrated bound follows from \eqref{eq:energy}.
\end{proof}

\begin{proposition}[Propagation of the positive cone]
\label{prop:cone}
Let \(A\) solve \eqref{eq:conf} on \([0,T]\), with
\(A(0,x)\in\Spp^3\) and \(\norm{\Log A(0)}_{L^\infty}<\infty\).  If
\[
  \int_0^T\norm{\nabla u(t)}_{L^\infty}\,dt<\infty,
\]
then \(A(t,x)\in\Spp^3\) persists and
\[
  \sup_{0\le t\le T}\norm{\Log A(t)}_{L^\infty}<\infty .
\]
The bound depends only on the initial spectral envelope, \(T,\lambda\), and
\(\int_0^T\norm{\nabla u}_{L^\infty}\,dt\).
\end{proposition}

\begin{proof}
By Lemma~\ref{lem:symmetry}, \(A\) remains symmetric.  Along particle
trajectories, the extremal eigenvalues obey, in the sense of upper and lower
Dini derivatives,
\[
  \frac{d}{dt}\lambda_{\max}A
  \le
  (2\norm{\nabla u}_{L^\infty}-\lambda^{-1})\lambda_{\max}A+\lambda^{-1},
\]
and
\[
  \frac{d}{dt}\lambda_{\min}A
  \ge
  -(2\norm{\nabla u}_{L^\infty}+\lambda^{-1})\lambda_{\min}A .
\]
These inequalities follow from the variational characterization of the
largest and smallest eigenvalues applied to
\(\dot A=\nabla u\,A+A(\nabla u)^T-\lambda^{-1}(A-\Id)\) along the flow.
The lower inequality gives the explicit comparison
\[
  \lambda_{\min}A(t,X(t;x))
  \ge
  \lambda_{\min}A_0(x)
  \exp\left(
  -\int_0^t(2\norm{\nabla u(\tau)}_{L^\infty}+\lambda^{-1})\,d\tau
  \right),
\]
so no eigenvalue can reach zero on \([0,T]\).  The upper inequality and
Gronwall's lemma give a finite upper bound depending on the same quantities.
These upper and lower spectral bounds are equivalent to a finite
\(L^\infty\) bound for \(\Log A\).
\end{proof}

\section{Endpoint vorticity clock}

Fix a smooth periodic Littlewood--Paley decomposition
\[
  f=\sum_{j\ge -1}\Delta_jf,\qquad
  \norm{f}_{\Besov}=\sum_{j\ge -1}\norm{\Delta_jf}_{L^\infty}.
\]

\begin{lemma}[Endpoint Biot--Savart estimate]
\label{lem:biot}
For divergence-free \(u\) on \(\T^3\), with \(\omega=\curl u\),
\begin{equation}
  \norm{\nabla u}_{\Besov}
  +
  \norm{\nabla u}_{L^\infty}
  \le C\norm{\omega}_{\Besov}.
  \label{eq:biot}
\end{equation}
\end{lemma}

\begin{proof}
The spatial mean of \(u\) is irrelevant for \(\nabla u\), so subtract it.  The
zero Fourier mode of \(\omega=\curl u\) vanishes, and for every nonzero mode
\[
  \widehat u(k)=\frac{i\,k\times\widehat\omega(k)}{|k|^2}.
\]
Thus each component of \(\nabla u\) is a zero-order Fourier multiplier applied
to \(\omega\), with the multiplier defined to be zero at \(k=0\).  The
inhomogeneous low block contains only finitely many Fourier modes and is
covered by the same bounded symbol estimate.  By
Lemma~\ref{lem:appendix-multiplier}, zero-order periodic multipliers are
bounded on \(B^0_{\infty,1}\), which gives the Besov estimate.  The
\(L^\infty\) estimate then follows from the endpoint embedding
\(\norm{g}_{L^\infty}\le\norm{g}_{B^0_{\infty,1}}\).
\end{proof}

\begin{corollary}
\label{cor:flow-cone}
For a smooth solution of \eqref{eq:momentum}--\eqref{eq:conf} on \([0,T]\),
if
\[
  \int_0^T\norm{\omega(t)}_{\Besov}\,dt<\infty,
\]
then \(\int_0^T\norm{\nabla u(t)}_{L^\infty}\,dt<\infty\).  Hence the
Lagrangian flow is bi-Lipschitz on \([0,T]\), and the positive-cone spectral
envelope of \(A\) is bounded on \([0,T]\).
\end{corollary}

\begin{proof}
The first assertion follows by integrating the \(L^\infty\) part of
Lemma~\ref{lem:biot}.  The ODE for the flow map then gives the standard
bi-Lipschitz bound
\[
  \norm{D_xX(t,\cdot)}_{L^\infty}
  +\norm{D_xX(t,\cdot)^{-1}}_{L^\infty}
  \le
  \exp\left(C\int_0^t\norm{\nabla u(\tau)}_{L^\infty}\,d\tau\right).
\]
The spectral bound for \(A\) is Proposition~\ref{prop:cone}.
\end{proof}

\section{Logarithmic conformation estimates}

Throughout this section \(s>5/2\).  On a compact spectral envelope
\(\norm{B}_{L^\infty}\le K\), matrix functions are understood pointwise by
spectral calculus, equivalently as smooth maps on the finite-dimensional
space \(\Sym^3\).  The compact-range composition lemma in
Appendix~\ref{app:aux} gives
\begin{equation}
  \norm{e^B-\Id}_{H^s}\le C_{K,s}\norm{B}_{H^s},
  \qquad
  \norm{B}_{H^s}\le C_{K,s}\norm{e^B-\Id}_{H^s}.
  \label{eq:matrix-moser}
\end{equation}
Indeed, the first bound is the vanishing composition estimate for
\(\mathcal F(B)=e^B-\Id\) on \(\{|B|\le K\}\).  For the reverse bound, write
\(C=e^B-\Id\).  The range of \(C\) lies in the compact set
\(\{e^X-\Id:\ |X|\le K\}\), where \(I+C\) is uniformly positive, and
\(\mathcal G(C)=\Log(I+C)\) is smooth and vanishes at \(C=0\).  Applying the
same compact-range estimate to \(\mathcal G\) gives the second inequality.

\begin{lemma}[Frechet logarithm on a compact cone]
\label{lem:frechet-log}
For \(A\in\Spp^3\) and symmetric \(H\),
\begin{equation}
  d\Log_A(H)=
  \int_0^\infty (A+r\Id)^{-1}H(A+r\Id)^{-1}\,dr .
  \label{eq:loewner}
\end{equation}
If \(A=e^B\), \(|B|\le K\), then
\begin{equation}
  \norm{d\Log_{e^B}(H)}_{H^s}
  \le
  C_{K,s}\bigl(\norm{H}_{H^s}+\norm{B}_{H^s}\norm{H}_{L^\infty}\bigr).
  \label{eq:frechet-tame}
\end{equation}
Moreover,
\begin{equation}
  d\Log_A(A-\Id)=\Id-A^{-1},
  \label{eq:log-relax-exact}
\end{equation}
and
\[
  (\Id-e^{-B}):B\ge c_K|B|^2,\qquad
  |\Id-e^{-B}|\le C_K|B|.
\]
\end{lemma}

\begin{proof}
Formula \eqref{eq:loewner} is the Loewner representation of the matrix
logarithm on the positive cone.  On \(e^{-K}\Id\le A\le e^K\Id\), the
resolvent factors are uniformly controlled.  For the Sobolev estimate, the
finite-dimensional map
\[
  (B,H)\mapsto d\Log_{e^B}(H)
\]
is smooth on \(\{|B|\le K\}\) and linear in \(H\).  Decompose it as
\[
  d\Log_{e^B}(H)=H+\mathcal M(B)H,\qquad \mathcal M(0)=0 .
\]
The \(L^\infty\) version of the product estimate in
Lemma~\ref{lem:appendix-product}, together with compact-range composition,
gives \eqref{eq:frechet-tame}.  To prove \eqref{eq:log-relax-exact},
diagonalize
\(A=Q\operatorname{diag}(a_i)Q^T\).  Since \(A-\Id\) commutes with \(A\), the
derivative acts on eigenvalues as \(d(\log a_i)(a_i-1)=(a_i-1)/a_i\).
The coercivity is the scalar inequality
\((1-e^{-b})b\ge c_K b^2\) on \([-K,K]\).
\end{proof}

\begin{lemma}[Endpoint transport commutator]
\label{lem:transport}
Let \(s>5/2\), \(\divv u=0\), and \(F:\T^3\to\R^m\) be smooth.  Then
\begin{equation}
  \left|
  \int_{\T^3}\Lambda^s(u\cdot\nabla F)\cdot\Lambda^sF\,dx
  \right|
  \le
  C_s\norm{\nabla u}_{B^0_{\infty,1}}\norm{F}_{H^s}^2 .
  \label{eq:transport}
\end{equation}
Consequently,
\begin{equation}
  \left|
  \int_{\T^3}\Lambda^s(u\cdot\nabla F)\cdot\Lambda^sF\,dx
  \right|
  \le
  C_s\norm{\omega}_{\Besov}\norm{F}_{H^s}^2 .
  \label{eq:transport-vorticity}
\end{equation}
\end{lemma}

\begin{proof}
Because \(\divv u=0\), the principal transport contribution
\(\int u\cdot\nabla\Lambda^sF\cdot\Lambda^sF\,dx\) vanishes.  The remaining
term is the commutator \([\Lambda^s,u\cdot\nabla]F\).  Bony's decomposition
and the standard dyadic commutator estimate yield
\[
  \norm{[\Delta_j,u\cdot\nabla]F}_{L^2}
  \le
  c_j\,2^{-js}
  \norm{\nabla u}_{B^0_{\infty,1}}\norm{F}_{H^s},
  \qquad (c_j)\in\ell^2 .
\]
Multiplying by \(2^{js}\), summing, and pairing with \(\Delta_jF\) gives
\eqref{eq:transport}.  Estimate \eqref{eq:transport-vorticity} follows from
Lemma~\ref{lem:biot}.
\end{proof}

Define
\[
  \mathfrak L(B,u)=
  d\Log_{e^B}\left(\nabla u\,e^B+e^B(\nabla u)^T-\lambda^{-1}(e^B-\Id)\right).
\]
Then \(B=\Log A\) satisfies
\[
  \partial_tB+u\cdot\nabla B=\mathfrak L(B,u).
\]

\begin{lemma}[Tame logarithmic source estimate]
\label{lem:log-source}
If \(\norm{B}_{L^\infty}\le K\), then
\begin{equation}
  \norm{\mathfrak L(B,u)}_{H^s}
  \le
  C_{K,s}
  \left[
  \norm{u}_{H^{s+1}}
  +(1+\norm{\omega}_{\Besov})\norm{B}_{H^s}
  \right].
  \label{eq:log-source}
\end{equation}
\end{lemma}

\begin{proof}
Split \(\mathfrak L\) into stretching and relaxation parts.  Since
\(d\Log_{\Id}\) is the identity,
\[
  d\Log_{e^B}\left(\nabla u\,e^B+e^B(\nabla u)^T\right)
  =
  \nabla u+(\nabla u)^T+\mathcal Q(B)\nabla u ,
\]
where \(\mathcal Q(0)=0\) and \(\mathcal Q\) is a smooth finite-dimensional
matrix-valued function on the compact set \(\{|B|\le K\}\).  The first two
terms are bounded by \(C\norm{u}_{H^{s+1}}\).  The remainder is controlled by
Lemma~\ref{lem:appendix-product} with \(Z=\nabla u\) and
\(X=B^0_{\infty,1}\):
\[
  \norm{\mathcal Q(B)\nabla u}_{H^s}
  \le
  C_{K,s}\left(
  \norm{u}_{H^{s+1}}
  +\norm{\nabla u}_{B^0_{\infty,1}}\norm{B}_{H^s}
  \right).
\]
For the relaxation term, \eqref{eq:log-relax-exact} gives exactly
\[
  -\lambda^{-1}d\Log_{e^B}(e^B-\Id)
  =
  -\lambda^{-1}(\Id-e^{-B}),
\]
which is bounded by \(C_{K,s}\norm{B}_{H^s}\) by compact-range composition.
Finally, Lemma~\ref{lem:biot} gives
\(\norm{\nabla u}_{B^0_{\infty,1}}\le C\norm{\omega}_{\Besov}\), completing
the proof.
\end{proof}

\begin{proposition}[Logarithmic high-order estimate]
\label{prop:log-Hs}
On any interval where \(\norm{B}_{L^\infty}\le K\),
\begin{equation}
  \frac{d}{dt}\norm{B}_{H^s}^2
  \le
  \frac{\nu}{4}\norm{u}_{H^{s+1}}^2
  +
  C_{K,s}(1+\norm{\omega}_{\Besov})(1+\norm{B}_{H^s}^2).
  \label{eq:log-Hs}
\end{equation}
\end{proposition}

\begin{proof}
Apply \(\Lambda^s\) to the \(B\)-equation and test by \(\Lambda^sB\).  The
transport part is bounded by Lemma~\ref{lem:transport}; the source part is
bounded by Lemma~\ref{lem:log-source}.  Young's inequality gives
\[
  C\norm{u}_{H^{s+1}}\norm{B}_{H^s}
  \le
  \frac{\nu}{4}\norm{u}_{H^{s+1}}^2+C\norm{B}_{H^s}^2 .
\]
Inserting this bound gives \eqref{eq:log-Hs}.  The velocity dissipation is
therefore recovered only after this estimate is combined with the velocity
energy inequality below.
\end{proof}

\begin{proposition}[Velocity high-order estimate]
\label{prop:u-Hs}
On the same spectral envelope,
\begin{equation}
  \frac{d}{dt}\norm{u}_{H^s}^2+c\nu\norm{u}_{H^{s+1}}^2
  \le
  C_{K,s}(1+\norm{\omega}_{\Besov})
  (1+\norm{u}_{H^s}^2+\norm{B}_{H^s}^2).
  \label{eq:u-Hs}
\end{equation}
\end{proposition}

\begin{proof}
Apply \(\Lambda^s\) to the velocity equation and test by \(\Lambda^su\).  The
transport term is bounded by \(\norm{\nabla u}_{L^\infty}\norm{u}_{H^s}^2\),
then by Lemma~\ref{lem:biot}.  The stress term satisfies, by
\eqref{eq:matrix-moser},
\[
  \left|
  \langle\Lambda^s\divv(e^B-\Id),\Lambda^su\rangle
  \right|
  \le
  \frac{\nu}{2}\norm{\nabla u}_{H^s}^2
  +C_{K,s}\norm{B}_{H^s}^2 .
\]
Strictly speaking, the viscous term gives
\(\nu\norm{\nabla u}_{H^s}^2\), which controls
\(\norm{u-\bar u}_{H^{s+1}}^2\) on the torus.  The spatial mean
\(\bar u\) is conserved, and its contribution to the full
\(H^{s+1}\)-norm is absorbed into the harmless constant in the right-hand
side.  This yields the stated estimate with a universal \(c>0\).
\end{proof}

\section{Proof of the geometric criterion}

\begin{proposition}[Local continuation framework]
\label{prop:local}
Let \(s>5/2\) and \(T<\infty\).  If a strong solution on \([0,T)\) satisfies
\[
  \sup_{t<T}\left(
  \norm{u(t)}_{H^s}
  +\norm{A(t)-\Id}_{H^s}
  +\norm{\Log A(t)}_{L^\infty}
  \right)<\infty,
\]
then it can be continued beyond \(T\).
\end{proposition}

\begin{proof}
Let
\[
  M=\sup_{t<T}\bigl(\norm{u(t)}_{H^s}+\norm{A(t)-\Id}_{H^s}\bigr),
  \qquad
  K=\sup_{t<T}\norm{\Log A(t)}_{L^\infty}.
\]
Then \(e^{-K}\Id\le A(t)\le e^K\Id\) for all \(t<T\).  By the uniform restart
statement, Lemma~\ref{lem:appendix-restart}, there is
\(\tau=\tau(M,K,\nu,\alpha,\lambda,s)>0\) such that every datum taken from the
trajectory at a time \(t<T\) generates a strong solution on
\([t,t+\tau]\).  Choose \(t_0<T\) with \(T-t_0<\tau/2\).  Restarting from
\((u(t_0),A(t_0))\) and using uniqueness on the overlap gives an extension of
the original solution to \([0,t_0+\tau]\), hence beyond \(T\).
\end{proof}

\begin{proof}[Proof of Theorem~\ref{thm:main}]
Assume that the two continuation quantities are finite on \([0,T)\).  Put
\(B=\Log A\) and
\[
  K=\sup_{0\le t<T}\norm{B(t)}_{L^\infty}<\infty .
\]
For every \(T_0<T\), adding \eqref{eq:log-Hs} and \eqref{eq:u-Hs} gives
\[
  \frac{d}{dt}Y(t)+c\nu\norm{u(t)}_{H^{s+1}}^2
  \le
  C_{K,s}(1+\norm{\omega(t)}_{\Besov})(1+Y(t)),
\]
where \(Y(t)=\norm{u(t)}_{H^s}^2+\norm{B(t)}_{H^s}^2\).  Gronwall's inequality
gives
\[
  \sup_{0\le t\le T_0}Y(t)
  +
  c\nu\int_0^{T_0}\norm{u(t)}_{H^{s+1}}^2\,dt
  \le
  C_{K,s,T,Y(0)}
  \exp\left(
    C_{K,s}\int_0^T\norm{\omega(t)}_{\Besov}\,dt
  \right),
\]
where the right-hand side is independent of \(T_0\).  Letting
\(T_0\uparrow T\) gives a uniform bound on \([0,T)\), and also
\(u\in L^2(0,T;H^{s+1})\).  By
\eqref{eq:matrix-moser},
\(A-\Id\) is bounded in \(H^s\), and the spectral envelope keeps \(A\)
uniformly positive.  Proposition~\ref{prop:local} restarts the solution beyond
\(T\).  Applying this with \(T\uparrow T_*\) proves the theorem.
\end{proof}

\begin{proof}[Proof of Corollary~\ref{cor:vorticity-only}]
Suppose, to the contrary, that
\(\int_0^{T_*}\norm{\omega(t)}_{\Besov}\,dt<\infty\).  For each
\(T_0<T_*\), Corollary~\ref{cor:flow-cone} and Proposition~\ref{prop:cone}
give a logarithmic spectral-envelope bound on \([0,T_0]\).  The bound depends
only on the initial envelope, \(T_*,\lambda\), and the total endpoint clock on
\([0,T_*)\), so it is independent of \(T_0\).  Thus
\(\sup_{t<T_*}\norm{\Log A(t)}_{L^\infty}<\infty\), contradicting
Theorem~\ref{thm:main}.
\end{proof}

\section{Good unknown and vorticity--stress coupling}

Let
\[
  \Sigma=A-\Id,\qquad
  \Gamma=\omega-\frac{\alpha}{\nu}\calR\Sigma,\qquad
  \calR=(-\Delta)^{-1}\curl\divv .
\]
Here \(\divv\) acts row-wise on matrix fields and \(\curl\) acts on the
resulting vector field.  The inverse \((-\Delta)^{-1}\) is taken on zero-mean
vector fields; constant tensor modes are removed by the divergence.
Equivalently, for \(k\ne0\),
\[
  \widehat{\calR F}(k)
  =
  |k|^{-2}\,ik\times\bigl(i\widehat F(k)k\bigr),
  \qquad
  \widehat{\calR F}(0)=0 .
\]
Thus \(\calR\) is a homogeneous order-zero multiplier on nonzero modes.

\begin{proposition}[Vorticity--stress coupling]
\label{prop:vorticity-stress}
The vorticity satisfies
\begin{equation}
  \partial_t\omega+u\cdot\nabla\omega-\nu\Delta\omega
  =
  \omega\cdot\nabla u+\alpha\,\curl\divv\Sigma .
  \label{eq:vorticity}
\end{equation}
Moreover,
\begin{equation}
\begin{aligned}
  (\partial_t+u\cdot\nabla-\nu\Delta)\Gamma
  &=
  \omega\cdot\nabla u
  -\frac{\alpha}{\nu}\calR
  \left(\nabla u\,A+A(\nabla u)^T-\lambda^{-1}\Sigma\right)\\
  &\quad
  -\frac{\alpha}{\nu}[u\cdot\nabla,\calR]\Sigma .
\end{aligned}
  \label{eq:Gamma}
\end{equation}
\end{proposition}

\begin{proof}
Taking curl of the momentum equation and using
\(\curl(u\cdot\nabla u)=u\cdot\nabla\omega-\omega\cdot\nabla u\) gives
\eqref{eq:vorticity}.  With our convention
\[
  (-\Delta)\calR\Sigma=\curl\divv\Sigma,
\]
so
\[
  -\nu\Delta\left(\frac{\alpha}{\nu}\calR\Sigma\right)
  =
  \alpha\curl\divv\Sigma .
\]
This is exactly the direct stress forcing in \eqref{eq:vorticity}, and it is
subtracted in \((\partial_t+u\cdot\nabla-\nu\Delta)\Gamma\).  More explicitly,
\[
\begin{aligned}
(\partial_t+u\cdot\nabla-\nu\Delta)\Gamma
&=(\partial_t+u\cdot\nabla-\nu\Delta)\omega\\
&\quad
-\frac{\alpha}{\nu}
(\partial_t+u\cdot\nabla-\nu\Delta)\calR\Sigma ,
\end{aligned}
\]
and the term \(\alpha\curl\divv\Sigma\) from the first line cancels the
diffusive contribution from the second.  Moreover,
\[
  \partial_t\calR\Sigma+u\cdot\nabla\calR\Sigma
  =
  \calR(\partial_t\Sigma+u\cdot\nabla\Sigma)
  +[u\cdot\nabla,\calR]\Sigma .
\]
Inserting
\[
  \partial_t\Sigma+u\cdot\nabla\Sigma
  =
  \nabla u\,A+A(\nabla u)^T-\lambda^{-1}\Sigma
\]
gives \eqref{eq:Gamma}.
\end{proof}

\begin{corollary}[Good-unknown clock]
\label{cor:good}
If \(T_*<\infty\), then
\[
  \int_0^{T_*}
  \left(
  \norm{\Gamma(t)}_{\Besov}
  +\norm{\Sigma(t)}_{\Besov}
  \right)\,dt=\infty .
\]
\end{corollary}

\begin{proof}
The operator \(\calR\) is a zero-order multiplier on \(B^0_{\infty,1}\), hence
\[
  \norm{\omega}_{\Besov}
  \le C(\norm{\Gamma}_{\Besov}+\norm{\Sigma}_{\Besov}).
\]
The conclusion follows from Corollary~\ref{cor:vorticity-only}.
\end{proof}

\begin{remark}[Effective vorticity and the remaining obstruction]
The variable \(\Gamma\) is an effective vorticity: it subtracts the part of the
vorticity generated directly by the polymer stress through the elliptic
operator \((-\Delta)^{-1}\curl\divv\).  Thus the derivative-bearing force
\(\alpha\curl\divv\Sigma\) in \eqref{eq:vorticity} is not left as an
uncontrolled source.  What remains in \eqref{eq:Gamma} is the genuine
three-dimensional obstruction.  In two dimensions the vortex stretching term is
absent.  In three dimensions, \(\omega\cdot\nabla u\) persists; the good
unknown isolates this Euler-type term but does not remove it.
\end{remark}

\section{Positive-cone Reynolds states}

Let \(B(t,x)\in\Sym^3\), \(A=e^B\).  A positive-cone Oldroyd--B Reynolds state
is a tuple
\[
  (u,p,B,R,S),\qquad R,S\in\Sym^3,\qquad \divv u=0,
\]
satisfying
\begin{align}
  \partial_t u+\divv(u\otimes u)-\nu\Delta u+\nabla p
  &=
  \alpha\divv(A-\Id)+\divv R,
  \label{eq:reynolds-momentum}\\
  \partial_t A+u\cdot\nabla A-\nabla u\,A-A(\nabla u)^T+\lambda^{-1}(A-\Id)
  &=S .
  \label{eq:reynolds-conf}
\end{align}

\begin{proposition}[Defect work identity]
\label{prop:defect-work}
Every smooth positive-cone Reynolds state satisfies
\begin{equation}
\begin{aligned}
  \frac{d}{dt}\calE(t)
  &+\nu\int_{\T^3}|\nabla u|^2\,dx
  +\frac{\alpha}{2\lambda}
  \int_{\T^3}\tr(A+A^{-1}-2\Id)\,dx\\
  &=
  -\int_{\T^3}R:D(u)\,dx
  +\frac{\alpha}{2}\int_{\T^3}(\Id-A^{-1}):S\,dx .
\end{aligned}
  \label{eq:defect-work}
\end{equation}
\end{proposition}

\begin{proof}
Testing \eqref{eq:reynolds-momentum} by \(u\) gives
\[
  \frac12\frac{d}{dt}\int_{\T^3}|u|^2\,dx
  +\nu\int_{\T^3}|\nabla u|^2\,dx
  =
  -\alpha\int_{\T^3}(A-\Id):\nabla u\,dx
  -\int_{\T^3}R:\nabla u\,dx .
\]
The transport and pressure terms vanish by incompressibility and periodicity.
Since \(R\) is symmetric,
\[
  R:\nabla u=R:D(u).
\]
For the conformation equation, testing \eqref{eq:reynolds-conf} by
\(D\Phi_3(A)=\Id-A^{-1}\) gives
\[
  \frac{d}{dt}\int_{\T^3}\Phi_3(A)\,dx
  =
  2\int_{\T^3}(A-\Id):\nabla u\,dx
  -\lambda^{-1}\int_{\T^3}\tr(A+A^{-1}-2\Id)\,dx
  +\int_{\T^3}(\Id-A^{-1}):S\,dx .
\]
The stretching term uses Lemma~\ref{lem:cancellation}.  Multiplying the last
display by \(\alpha/2\) and adding it to the kinetic balance cancels the
stress work and leaves exactly \eqref{eq:defect-work}.
\end{proof}

\begin{definition}[Energy--entropy admissibility]
A Reynolds state is energy--entropy admissible on a time interval if
\begin{equation}
  \int_{\T^3}
  \left[
  -R:D(u)+\frac{\alpha}{2}(\Id-A^{-1}):S
  \right]dx
  \le0
  \quad\hbox{for a.e. }t.
  \label{eq:admissibility}
\end{equation}
\end{definition}

\subsection{Symmetric trace-free anti-divergence}

\begin{proposition}[Anti-divergence in three dimensions]
\label{prop:antidiv}
Let \(m:\T^3\to\R^3\) be smooth with zero mean.  Then there exists a smooth
symmetric trace-free zero-mean tensor \(R_m\) such that
\[
  \divv R_m=m.
\]
Moreover, for every integer \(k\ge0\),
\[
  \norm{R_m}_{H^{k+1}}\le C_k\norm{m}_{H^k}.
\]
\end{proposition}

\begin{proof}
The zero-mean condition is necessary, since the divergence of a periodic
tensor has zero spatial mean.  Define the zero-mean vector field \(\phi\) by
the elliptic system
\[
  \Delta\phi+\frac13\nabla\divv\phi=m .
\]
Equivalently, for \(k\in\mathbb Z^3\setminus\{0\}\),
\[
  \widehat\phi(k)
  =
  -\left(|k|^2\Id+\frac13 k\otimes k\right)^{-1}\widehat m(k),
  \qquad
  \widehat\phi(0)=0 .
\]
The matrix in parentheses has eigenvalues \(|k|^2\) on \(k^\perp\) and
\((4/3)|k|^2\) in the \(k\)-direction, so the inverse is a homogeneous
multiplier of order \(-2\) on nonzero modes.  Set
\[
  (R_m)_{ij}
  =
  \partial_i\phi_j+\partial_j\phi_i-\frac23\delta_{ij}\divv\phi .
\]
Then \(R_m\) is symmetric, trace free, and has zero spatial mean.  Moreover,
\[
  (\divv R_m)_i
  =
  \Delta\phi_i+\frac13\partial_i\divv\phi=m_i .
\]
Since \(R_m\) is obtained by applying one derivative to a multiplier of order
\(-2\), \(\calA:m\mapsto R_m\) is a Fourier multiplier of order \(-1\).  This
gives \(\norm{R_m}_{H^{k+1}}\le C_k\norm{m}_{H^k}\).  We denote this fixed
linear representative by
\[
  \calA m:=R_m .
\]
\end{proof}

\subsection{Pressure-free residual work}

For smooth \(u,p,B\), with \(\divv u=0\) and \(B(t,x)\in\Sym^3\), define
\[
  \calM[u,p,B]
  =
  \partial_t u+\divv(u\otimes u)-\nu\Delta u+\nabla p
  -\alpha\divv(A-\Id),
\]
and
\[
  \calS[u,B]
  =
  \partial_t A+u\cdot\nabla A-\nabla u\,A-A(\nabla u)^T+\lambda^{-1}(A-\Id).
\]
For \(B\in\Sym^3\), this residual is symmetric: the only nontrivial point is
that \((\nabla u\,A)^T=A(\nabla u)^T\), so
\(\nabla u\,A+A(\nabla u)^T\in\Sym^3\).
Assume the spatial mean of \(u\) is fixed in time, and define the
pressure-free residual
\[
  \calN[u,B]
  =
  \partial_t u+\divv(u\otimes u)-\nu\Delta u-\alpha\divv(A-\Id).
\]
Then \(\calN[u,B]\) has zero spatial mean: the divergence terms integrate to
zero on \(\T^3\), and the fixed spatial mean gives
\(\int_{\T^3}\partial_t u\,dx=0\).  Thus Proposition~\ref{prop:antidiv}
gives the canonical symmetric trace-free tensor
\[
  R_{\calN}=\calA\calN[u,B],
  \qquad
  \divv R_{\calN}=\calN[u,B].
\]
All pressure-free residual-work quantities below are computed with this fixed
representative.  The full residual
\(\calM[u,p,B]=\calN[u,B]+\nabla p\) is represented by
\(R_{\calN}+p\Id\), because \(\divv(p\Id)=\nabla p\).  This representative is
not trace free unless \(p\equiv0\), but the trace part is precisely the
pressure gauge and performs no work:
\[
  \int_{\T^3}p\Id:D(u)\,dx
  =
  \int_{\T^3}p\,\divv u\,dx=0 .
\]
Hence admissibility is equivalent to
\begin{equation}
  \int_{\T^3}
  \left[
  -R_{\calN}:D(u)
  +\frac{\alpha}{2}(\Id-A^{-1}):\calS[u,B]
  \right]dx
  \le0 .
  \label{eq:pressure-free}
\end{equation}

\section{Sharp residual-work cost}

Set
\[
  G(t)=\Id-A(t)^{-1},\qquad
  P(t)=\int_{\T^3}-R_{\calN}(t):D(u(t))\,dx .
\]

\begin{proof}[Proof of Theorem~\ref{thm:residual-main}]
Admissibility in the pressure-free form \eqref{eq:pressure-free} reads
\[
  P(t)+\frac{\alpha}{2}\int_{\T^3}G(t):\calS[u,B](t)\,dx\le0 .
\]
If \(P(t)>0\), Cauchy's inequality gives
\[
  P(t)\le
  \frac{\alpha}{2}\norm{G(t)}_{L^2}\norm{\calS[u,B](t)}_{L^2},
\]
which proves \eqref{eq:main-residual-cost}.  If \(G(t)\equiv0\), the same
admissibility inequality forces \(P(t)\le0\).

For the sharp cost, fix \(t\) and minimize over all candidate defects \(S\)
subject to the same scalar work constraint.  Decompose \(S=aG+S^\perp\) with
\(\int G:S^\perp\,dx=0\).  The work constraint is
\[
  P(t)+\frac{\alpha}{2}a\norm{G(t)}_{L^2}^2\le0,
\]
and
\[
  \norm{S}_{L^2}^2=a^2\norm{G(t)}_{L^2}^2+\norm{S^\perp}_{L^2}^2.
\]
For \(P(t)>0\), the unique minimizer is obtained by taking
\[
  a=-\frac{2P(t)}{\alpha\norm{G(t)}_{L^2}^2},
  \qquad S^\perp=0.
\]
This gives \eqref{eq:main-min-cost}.  For \(P(t)\le0\), the minimum is
attained by \(S=0\).
\end{proof}

\begin{corollary}[Integrated three-channel budget]
\label{cor:three-channel}
For every measurable time window \(E\), define
\[
  W_E=\int_E[P(t)]_+\,dt,\quad
  L_E=\norm{G}_{L^2(E\times\T^3)},\quad
  S_E=\norm{\calS[u,B]}_{L^2(E\times\T^3)}.
\]
Let
\[
  \eta_E=
  \operatorname*{ess\,sup}_{t\in E}
  \frac{\left[-\int_{\T^3}G(t):\calS[u,B](t)\,dx\right]_+}
       {\norm{G(t)}_{L^2}\norm{\calS[u,B](t)}_{L^2}},
\]
with the ratio set equal to \(0\) when the denominator vanishes.  Then every
admissible residual-matching state satisfies
\begin{equation}
  W_E\le \frac{\alpha}{2}\eta_E L_ES_E .
  \label{eq:three-channel}
\end{equation}
\end{corollary}

\begin{proof}
By Cauchy's inequality, \(0\le\eta_E\le1\).  The pressure-free
admissibility inequality gives, for a.e. \(t\in E\),
\[
  [P(t)]_+
  \le
  \frac{\alpha}{2}
  \left[-\int_{\T^3}G(t):\calS[u,B](t)\,dx\right]_+
  \le
  \frac{\alpha}{2}\eta_E
  \norm{G(t)}_{L^2}\norm{\calS[u,B](t)}_{L^2}.
\]
Integrating over \(E\) and applying Cauchy's inequality in time gives
\eqref{eq:three-channel}.
\end{proof}

\begin{remark}[Sharpness and alignment]
The pointwise cost in Theorem~\ref{thm:residual-main} is saturated at a time
\(t\) with \(P(t)>0\) exactly when the conformation residual has no component
orthogonal to \(G(t)\) and is negatively aligned with it:
\[
  \calS[u,B](t,x)
  =
  -\frac{2P(t)}{\alpha\norm{G(t)}_{L^2}^2}G(t,x).
\]
Thus the three factors in \eqref{eq:three-channel} have distinct meanings:
\(L_E\) measures the size of the entropy lever, \(S_E\) measures the available
conformation residual, and \(\eta_E\) measures the negative alignment between
the two.  Near saturation of the integrated inequality can occur only when
the pointwise residuals are nearly antiparallel to \(G\) on the times where
\([P(t)]_+\) is significant and when the time Cauchy--Schwarz step is nearly
sharp.
\end{remark}

\section{Certificates and cone-tip sensitivity}

The integrated inequality \eqref{eq:three-channel} gives a computable
necessary condition.  If available estimates imply
\[
  W_E>\frac{\alpha}{2}\eta_*L_*S_*,
\]
while \(L_E\le L_*\), \(S_E\le S_*\), and \(\eta_E\le\eta_*\), then the
residual matching cannot be energy--entropy admissible.

\begin{proposition}[Near-equilibrium residual barrier]
\label{prop:near-equilibrium}
Assume on \(E\times\T^3\) that
\[
  \norm{B}_{L^\infty}\le\varepsilon,\qquad 0<\varepsilon\le1.
\]
Then
\[
  \norm{G}_{L^2(E\times\T^3)}
  \le e\,\varepsilon\,|E\times\T^3|^{1/2}.
\]
Consequently, every admissible residual-matching state satisfies
\[
  \int_E[P(t)]_+\,dt
  \le
  \frac{\alpha e}{2}
  \eta_E\,\varepsilon\,|E\times\T^3|^{1/2}
  \norm{\calS[u,B]}_{L^2(E\times\T^3)}.
\]
\end{proposition}

\begin{proof}
By spectral calculus and the scalar bound
\(|1-e^{-b}|\le e|b|\) for \(|b|\le1\), one has
\(|G|=|\Id-e^{-B}|\le e|B|\le e\varepsilon\).  The result follows from
\eqref{eq:three-channel}.
\end{proof}

\begin{corollary}[Cone-tip blow-up of the cost]
Let \(E_n\) be time windows with \(|E_n\times\T^3|\le M\), and suppose
\[
  \norm{B_n}_{L^\infty(E_n\times\T^3)}\le\varepsilon_n\to0,\qquad
  \int_{E_n}[P_n(t)]_+\,dt\ge W_0>0 .
\]
Then every admissible residual-matching state satisfies
\[
  \norm{\calS_n}_{L^2(E_n\times\T^3)}
  \ge
  \frac{2W_0}{\alpha e\,\varepsilon_nM^{1/2}}
  \to\infty .
\]
\end{corollary}

\begin{proof}
Apply Proposition~\ref{prop:near-equilibrium} on \(E_n\).  Since
\(0\le\eta_{E_n}\le1\) and \(|E_n\times\T^3|\le M\),
\[
  W_0
  \le
  \int_{E_n}[P_n(t)]_+\,dt
  \le
  \frac{\alpha e}{2}
  \varepsilon_n M^{1/2}
  \norm{\calS_n}_{L^2(E_n\times\T^3)} .
\]
Rearranging gives the stated lower bound, which diverges as
\(\varepsilon_n\to0\).
\end{proof}

\section{Maximal-time alternatives}

For a smooth solution on \([0,T)\), define
\[
  \mathcal S_T=\sup_{0\le t<T}\norm{\Log A(t)}_{L^\infty},\qquad
  \mathcal V_T=\int_0^T\norm{\omega(t)}_{\Besov}\,dt,
\]
and
\[
  \mathcal G_T=
  \int_0^T
  \left(\norm{\Gamma(t)}_{\Besov}+\norm{A(t)-\Id}_{\Besov}\right)\,dt .
\]

\begin{proposition}[Tail alternative]
If \(T_*<\infty\) is a maximal time of strong existence, then for every
\(t_0<T_*\) at least one of the following holds:
\[
  \sup_{t_0\le t<T_*}\norm{\Log A(t)}_{L^\infty}=\infty,
\]
or
\[
  \int_{t_0}^{T_*}\norm{\omega(t)}_{\Besov}\,dt=\infty .
\]
If the spectral tail is finite, then
\[
  \int_{t_0}^{T_*}
  \left(\norm{\Gamma(t)}_{\Besov}+\norm{A(t)-\Id}_{\Besov}\right)\,dt=\infty .
\]
\end{proposition}

\begin{proof}
Restart the strong solution at \(t_0\).  If both the spectral tail and the
vorticity tail were finite, Theorem~\ref{thm:main} would continue the solution
beyond \(T_*\), contradicting maximality.  If the spectral tail is finite but
the good-unknown tail were finite, the zero-order multiplier bound for
\(\calR\) would give
\[
  \int_{t_0}^{T_*}\norm{\omega(t)}_{\Besov}\,dt
  \le
  C\int_{t_0}^{T_*}
  \left(\norm{\Gamma(t)}_{\Besov}
  +\norm{A(t)-\Id}_{\Besov}\right)\,dt
  <\infty,
\]
again contradicting the first alternative.
\end{proof}

\section{Discussion}

The continuation criterion should be read as a positive-cone refinement of the
Beale--Kato--Majda mechanism, not as a formal restatement of it.  The endpoint
clock is used here simultaneously for three tasks: to control the Lipschitz
flow, to keep the logarithmic conformation equation within a compact spectral
envelope, and to commute the order-zero stress multipliers without stress
diffusion.  This coupling is the additional point in the Oldroyd--B setting.
The three-dimensional theory nevertheless retains the genuine vortex stretching
obstruction.  Thus the result is not a global regularity theorem.  Rather, it
identifies the endpoint quantity that any large-data regularity proof must
control:
\[
  \int_0^T\norm{\curl u(t)}_{B^0_{\infty,1}}\,dt.
\]
The endpoint nature of this quantity is essential: it is exactly strong enough
to control the Lipschitz flow and the order-zero stress multipliers, but weak
enough to avoid imposing a full Sobolev bound on the velocity gradient.

The comparison with stress-diffusive models is instructive.  If a Laplacian is
added to the conformation equation, high frequencies of \(A\) acquire their own
parabolic smoothing and the derivative carried by \(\divv A\) can be balanced
against stress diffusion.  In the present model there is no such smoothing
reservoir.  The velocity equation is parabolic, but the elastic variable is
transported and stretched, so the proof must trade only on the velocity
dissipation, relaxation, endpoint flow control, and the logarithmic geometry of
\(\Spp^3\).

At the relaxed level, the same geometry appears as a cost principle rather
than as a continuation criterion.  Positive pressure-free residual production
must be paid by a nonzero entropy-dual lever \(G=\Id-A^{-1}\), by a sufficiently
large conformation residual, and by the correct negative alignment.  Near
\(A=\Id\), the lever vanishes and the cost of paying fixed positive work
diverges.  Thus a relaxed continuation beyond the strong regime cannot move
mechanical work into an arbitrary Reynolds defect while keeping the
conformation equation energetically passive.  In numerical or relaxed-model
contexts, the criterion gives a test for closures: a residual closure with
small conformation defect and small entropy-dual lever cannot be
energy--entropy admissible if it produces positive pressure-free work.

\section{Conclusion}

The paper gives a three-dimensional positive-cone formulation of two necessary
constraints for stress-diffusion-free Oldroyd--B dynamics.  For strong
solutions, finite-time breakdown can occur only through loss of the logarithmic
spectral envelope of the conformation tensor or divergence of the endpoint
vorticity clock.  This conclusion is obtained in the log-conformation variable
and uses the endpoint Besov control needed to propagate the Lipschitz flow,
commute order-zero stress multipliers, and close the high-order estimates
without stress diffusion.

The residual-work result gives a complementary constraint at the level of
relaxed positive-cone states.  Positive pressure-free residual work cannot be
assigned freely: the least \(L^2\) conformation residual that can pay such work
is fixed by the entropy-dual lever \(G=\Id-A^{-1}\).  The resulting cone-tip
barrier shows that near the equilibrium \(A=\Id\), admissible positive work
requires a quantitatively large or strongly aligned conformation defect.

Taken together, the continuation and residual-work criteria locate the same
positive-cone obstruction on the two sides of possible breakdown.  Before
breakdown, one must control the endpoint flow clock on a compact logarithmic
cone.  After passing to a relaxed description, any positive pressure-free work
must be paid for by an exact entropy-dual conformation defect.  The results do
not remove the vortex stretching mechanism or prove large-data global
regularity.  They identify what such a proof, or any admissible relaxed
substitute for it, would still have to control.

\appendix

\section{Auxiliary analytic estimates}
\label{app:aux}

This appendix records the standard endpoint estimates used above.  They are
included to make clear that no stress diffusion or hidden parabolic smoothing
of the conformation tensor is being used.

\begin{lemma}[Zero-order multipliers on the endpoint clock]
\label{lem:appendix-multiplier}
Let \(T\) be a periodic Fourier multiplier whose symbol is smooth away from
the origin and satisfies the usual order-zero bounds on nonzero modes.  Define
the value at \(k=0\) arbitrarily, or set it to zero when the operator is
applied only to zero-mean data.  Then
\[
  \norm{Tf}_{B^0_{\infty,1}}
  \le C_T\norm{f}_{B^0_{\infty,1}} .
\]
In particular, the Riesz transforms and the operator
\((-\Delta)^{-1}\curl\divv\) are bounded on \(B^0_{\infty,1}\) when applied
to the zero-mean part of the input.
\end{lemma}

\begin{proof}
For \(j\ge0\), \(T\Delta_j\) has a convolution kernel with uniformly bounded
\(L^1\)-norm, by the standard rescaled Mihlin estimate.  The inhomogeneous
block \(j=-1\) involves only bounded frequencies and is controlled directly
by the finite-mode kernel.  Hence
\[
  \norm{\Delta_jTf}_{L^\infty}
  \le
  C_T\sum_{|k-j|\le C}\norm{\Delta_kf}_{L^\infty}.
\]
Summing in \(j\) gives the result.
\end{proof}

\begin{lemma}[Endpoint commutator estimate]
\label{lem:appendix-commutator}
Let \(s>5/2\), \(\divv u=0\), and \(F:\T^3\to\R^m\).  Then
\[
  \norm{[\Lambda^s,u\cdot\nabla]F}_{L^2}
  \le
  C_s\norm{\nabla u}_{B^0_{\infty,1}}\norm{F}_{H^s}.
\]
\end{lemma}

\begin{proof}
Write \(u\cdot\nabla F=T_u\nabla F+T_{\nabla F}u+R(u,\nabla F)\).  For the
principal low--high part, the kernel representation of \(\Delta_j\) gives
\[
  \norm{[\Delta_j,S_{j-1}u\cdot\nabla]\Delta_jF}_{L^2}
  \le
  C\norm{\nabla S_{j-1}u}_{L^\infty}\norm{\Delta_jF}_{L^2}.
\]
The low-frequency factor is bounded by
\(\norm{\nabla u}_{B^0_{\infty,1}}\).  For the high--low and high--high
pieces, Bernstein's inequality and almost orthogonality yield
\[
  \norm{\Delta_j(T_{\nabla F}u)}_{L^2}
  +
  \norm{\Delta_jR(u,\nabla F)}_{L^2}
  \le
  C\sum_{|k-j|\le C}
  \norm{\Delta_k\nabla u}_{L^\infty}
  \norm{\widetilde\Delta_kF}_{L^2}.
\]
Multiplication by \(2^{js}\), square summation in \(j\), and the summability of
the neighboring interactions give the asserted bound.
\end{proof}

\begin{lemma}[Endpoint product estimate]
\label{lem:appendix-product}
Let \(s>5/2\).  Let \(M:U\subset\R^N\to\R^{N_1\times N_2}\) be smooth and let
\(\mathcal K\Subset U\).  If \(B:\T^3\to\mathcal K\) and
\(Z:\T^3\to\R^{N_2}\) are smooth, then
\[
  \norm{M(B)Z}_{H^s}
  \le
  C_{\mathcal K,s,M}
  \left(
  \norm{Z}_{H^s}
  +\norm{Z}_{X}\norm{B}_{H^s}
  \right).
\]
Here \(X\) may be either \(L^\infty\) or \(B^0_{\infty,1}\).
\end{lemma}

\begin{proof}
Fix \(B_*\in\mathcal K\) and write
\[
  M(B)Z=M(B_*)Z+\bigl(M(B)-M(B_*)\bigr)Z .
\]
The constant-coefficient term is bounded by \(\norm{Z}_{H^s}\).  For the
remaining term use the tame product estimate
\[
  \norm{fg}_{H^s}
  \le C_s\left(
  \norm{f}_{L^\infty}\norm{g}_{H^s}
  +\norm{g}_{X}\norm{f}_{H^s}
  \right)
\]
with \(f=M(B)-M(B_*)\), \(g=Z\), and
\(X=L^\infty\) or \(B^0_{\infty,1}\).  The \(L^\infty\) case is the standard
Sobolev tame product estimate, while the endpoint Besov case follows from the
same paraproduct proof with the high-frequency coefficient measured in
\(B^0_{\infty,1}\).
For the composition bound, set \(V=B-B_*\) and
\(\widetilde M(V)=M(B_*+V)-M(B_*)\).  On the compact translated range
\(\mathcal K-B_*\), the map \(\widetilde M\) is smooth and
\(\widetilde M(0)=0\).  The compact-range composition estimate therefore gives
\[
  \norm{f}_{L^\infty}\le C_{\mathcal K,M},
  \qquad
  \norm{f}_{H^s}\le C_{\mathcal K,s,M}\norm{B-B_*}_{H^s}
  = C_{\mathcal K,s,M}\norm{B}_{H^s},
\]
which completes the proof.
\end{proof}

\begin{lemma}[Compact-range composition]
\label{lem:appendix-composition}
Let \(s>5/2\).  Let \(U\subset\R^N\) be open and let
\(\mathcal F:U\to\R^M\) be \(C^{[s]+2}\).  If \(V:\T^3\to U\) has range in a
compact set \(K\Subset U\), then
\[
  \norm{\mathcal F(V)}_{H^s}
  \le
  C_{K,s,\mathcal F}\bigl(1+\norm{V}_{H^s}\bigr).
\]
If additionally \(\mathcal F(0)=0\) and \(0\in K\), then
\[
  \norm{\mathcal F(V)}_{H^s}
  \le
  C_{K,s,\mathcal F}\norm{V}_{H^s}.
\]
\end{lemma}

\begin{proof}
This is the usual Moser composition estimate.  Since \(s>3/2\), \(H^s(\T^3)\)
is an algebra and embeds into \(L^\infty\).  The derivatives in the fractional
chain rule are evaluated only on the compact set \(K\), so all constants are
finite and depend only on \(K,s,\mathcal F\).  For the vanishing case, choose
a cutoff \(\chi\in C_c^\infty(U)\) equal to \(1\) on a neighborhood of \(K\).
The product \(\chi\mathcal F\), extended by zero outside \(U\), is a globally
smooth map on \(\R^N\), agrees with \(\mathcal F\) on the range of \(V\), and
still vanishes at the origin.  The standard paralinearized Moser estimate
therefore gives
\[
  \norm{\mathcal F(V)}_{H^s}
  \le
  C_{K,s,\mathcal F}\norm{V}_{H^s}.
\]
No star-shaped assumption on \(K\) is required.
\end{proof}

\begin{lemma}[Uniform restart time]
\label{lem:appendix-restart}
Let \(s>5/2\).  Consider initial data satisfying
\[
  \norm{u_0}_{H^s}+\norm{A_0-\Id}_{H^s}\le M,
  \qquad
  e^{-K}\Id\le A_0\le e^K\Id .
\]
Then the local strong existence time for \eqref{eq:momentum}--\eqref{eq:conf}
has a positive lower bound depending only on \(M,K,\nu,\alpha,\lambda,s\).
\end{lemma}

\begin{proof}
The Picard iteration is performed on a ball in
\[
  u\in C([0,\tau];H^s)\cap L^2(0,\tau;H^{s+1}),
  \qquad
  A-\Id\in C([0,\tau];H^s).
\]
For example, take the ball where
\[
  \sup_{0\le t\le\tau}
  \bigl(\norm{u(t)}_{H^s}+\norm{A(t)-\Id}_{H^s}\bigr)\le 2C_0M+1
\]
and require, in addition, the short-time spectral envelope
\[
  e^{-(K+1)}\Id\le A(t)\le e^{K+1}\Id .
\]
The Stokes estimate gives the velocity smoothing in this ball, while the
conformation equation is solved along the Lipschitz flow generated by the
current velocity iterate.  Since \(s>5/2\), \(H^s(\T^3)\hookrightarrow C^1\)
and \(H^s\) is an algebra; the maps
\[
  (u,A)\mapsto u\cdot\nabla u,\qquad
  (u,A)\mapsto \nabla u\,A+A(\nabla u)^T-\lambda^{-1}(A-\Id)
\]
are locally Lipschitz in the displayed norms on this ball.  The constants in
these Lipschitz estimates depend only on \(M,K,\nu,\alpha,\lambda,s\).

It remains only to keep the cone envelope uniform during the iteration.
For velocities in the ball,
\[
  \int_0^\tau\norm{\nabla u(t)}_{L^\infty}\,dt
  \le C_s\tau\sup_{0\le t\le\tau}\norm{u(t)}_{H^s}.
\]
Choosing \(\tau\) small in terms of the same parameters and applying
Proposition~\ref{prop:cone} keeps the iterates inside
\(e^{-(K+1)}\Id\le A\le e^{K+1}\Id\).  The same choice of \(\tau\), made
smaller if necessary, makes the Picard map a contraction and maps the ball
into itself.  Hence the local existence time has a positive lower bound
depending only on \(M,K,\nu,\alpha,\lambda,s\).
\end{proof}

\section*{Acknowledgements}

\begin{sloppypar}
The author acknowledges financial support from the National Natural Science
Foundation of China (NSFC, Grant No. 12501602); the Education Department of
Hunan Province (Grant No. 24C0055); the Science and Technology Department of
Hunan Province (Grant No. 2025JJ60052); and the Scientific Research Start-up
Fund of Xiangtan University (Grant No. KZ0810769).
\end{sloppypar}


\begin{thebibliography}{99}

\bibitem{BKM}
J.~T. Beale, T.~Kato, and A.~Majda.
\newblock Remarks on the breakdown of smooth solutions for the 3-D Euler
equations.
\newblock \emph{Communications in Mathematical Physics}, 94:61--66, 1984.

\bibitem{BarrettBoyaval}
J.~W. Barrett and S.~Boyaval.
\newblock Existence and approximation of a (regularized) Oldroyd-B model.
\newblock \emph{Mathematical Models and Methods in Applied Sciences},
21:1783--1837, 2011.

\bibitem{BCD}
H.~Bahouri, J.-Y. Chemin, and R.~Danchin.
\newblock \emph{Fourier Analysis and Nonlinear Partial Differential
Equations}.
\newblock Springer, 2011.

\bibitem{Bird}
R.~B. Bird, R.~C. Armstrong, and O.~Hassager.
\newblock \emph{Dynamics of Polymeric Liquids. Vol. 1: Fluid Mechanics}.
\newblock Wiley, second edition, 1987.

\bibitem{CheminMasmoudi}
J.-Y. Chemin and N.~Masmoudi.
\newblock About lifespan of regular solutions of equations related to
viscoelastic fluids.
\newblock \emph{SIAM Journal on Mathematical Analysis}, 33:84--112, 2001.

\bibitem{ChengLuoYangYuan}
X.~Cheng, Z.~Luo, Z.~Yang, and C.~Yuan.
\newblock Global well-posedness and uniform-in-time vanishing damping limit for
the inviscid Oldroyd-B model.
\newblock arXiv:2410.09340, 2024.

\bibitem{ConstantinKliegl}
P.~Constantin and M.~Kliegl.
\newblock Note on global regularity for two-dimensional Oldroyd-B fluids with
diffusive stress.
\newblock \emph{Archive for Rational Mechanics and Analysis}, 206:725--740,
2012.

\bibitem{DeLellisSzekelyhidi}
C.~De Lellis and L.~Szekelyhidi Jr.
\newblock The Euler equations as a differential inclusion.
\newblock \emph{Annals of Mathematics}, 170:1417--1436, 2009.

\bibitem{ElgindiRousset}
T.~M. Elgindi and F.~Rousset.
\newblock Global regularity for some Oldroyd-B type models.
\newblock \emph{Communications on Pure and Applied Mathematics},
68:2005--2021, 2015.

\bibitem{FattalKupferman}
R.~Fattal and R.~Kupferman.
\newblock Constitutive laws for the matrix-logarithm of the conformation tensor.
\newblock \emph{Journal of Non-Newtonian Fluid Mechanics}, 123:281--285, 2004.

\bibitem{GuillopeSaut}
C.~Guillope and J.-C. Saut.
\newblock Existence results for the flow of viscoelastic fluids with a
differential constitutive law.
\newblock \emph{Nonlinear Analysis}, 15:849--869, 1990.

\bibitem{HuangLiuZi}
J.~Huang, Q.~Liu, and R.~Zi.
\newblock Global existence and decay rates of solutions to the Oldroyd-B model
with stress tensor diffusion.
\newblock \emph{Journal of Differential Equations}, 389:38--89, 2024.

\bibitem{KatoPonce}
T.~Kato and G.~Ponce.
\newblock Commutator estimates and the Euler and Navier-Stokes equations.
\newblock \emph{Communications on Pure and Applied Mathematics},
41:891--907, 1988.

\bibitem{LeiMasmoudiZhou}
Z.~Lei, N.~Masmoudi, and Y.~Zhou.
\newblock Remarks on the blowup criteria for Oldroyd models.
\newblock \emph{Journal of Differential Equations}, 248:328--341, 2010.

\bibitem{LiangLiZhai}
T.~Liang, Y.~Li, and X.~Zhai.
\newblock Large global solutions to the Oldroyd-B model with dissipation.
\newblock arXiv:2504.12986, 2025.

\bibitem{LinWuBoardman}
H.~Lin, J.~Wu, and N.~Boardman.
\newblock Stability on a 3D incompressible Oldroyd-B model with mixed partial
dissipation.
\newblock \emph{Science China Mathematics}, 68:1567--1606, 2025.

\bibitem{LionsMasmoudi}
P.-L. Lions and N.~Masmoudi.
\newblock Global solutions for some Oldroyd models of non-Newtonian flows.
\newblock \emph{Chinese Annals of Mathematics. Series B}, 21:131--146, 2000.

\bibitem{Oldroyd}
J.~G. Oldroyd.
\newblock On the formulation of rheological equations of state.
\newblock \emph{Proceedings of the Royal Society of London. Series A},
200:523--541, 1950.

\bibitem{PengUnified}
S.~Peng.
\newblock Geometric blow-up criteria for viscoelastic flows: Oldroyd-B and
FENE-P models.
\newblock Manuscript, 2026.

\bibitem{Renardy}
M.~Renardy.
\newblock \emph{Mathematical Analysis of Viscoelastic Flows}.
\newblock SIAM, 2000.

\bibitem{RenardyThomases}
M.~Renardy and B.~Thomases.
\newblock A mathematician's perspective on the Oldroyd B model: Progress and
future challenges.
\newblock \emph{Journal of Non-Newtonian Fluid Mechanics}, 293:104573, 2021.

\bibitem{Zhu2018}
Y.~Zhu.
\newblock Global small solutions of 3D incompressible Oldroyd-B model without
damping mechanism.
\newblock \emph{Journal of Functional Analysis}, 274:2039--2060, 2018.

\bibitem{Zi2021}
R.~Zi.
\newblock Vanishing viscosity limit of the 3D incompressible Oldroyd-B model.
\newblock \emph{Annales de l'Institut Henri Poincare C, Analyse Non Lineaire},
38:1841--1867, 2021.

\end{thebibliography}
\end{document}